\documentclass[a4paper]{amsart}
\usepackage{graphicx}
\usepackage{amssymb}
\usepackage{amsmath}
\usepackage{amsthm}
\usepackage{amscd}
\usepackage[all,2cell]{xy}

\UseAllTwocells \SilentMatrices
\newtheorem{thm}{Theorem}[section]

\newtheorem{cor}[thm]{Corollary}
\newtheorem{lem}[thm]{Lemma}
\newtheorem{exm}[thm]{Example}

\newtheorem{prop}[thm]{Proposition}
\theoremstyle{definition}

\theoremstyle{remark}
\newtheorem{rem}[thm]{\bf Remark}
\numberwithin{equation}{section}

\begin{document}
\title[The singularity category of an algebra with radical square zero]
{The singularity category of an algebra with radical square zero}

\author[  Xiao-Wu Chen
] {Xiao-Wu Chen}

\thanks{The author is supported by Special Foundation of President of The Chinese Academy of Sciences
(No.1731112304061) and  National Natural Science Foundation of China (No.10971206).}
\date{April 20, 2011}

\thanks{E-mail:
xwchen$\symbol{64}$mail.ustc.edu.cn}

\maketitle

\dedicatory{}%
\commby{}%

\begin{abstract}
To an artin algebra with radical square zero, a regular algebra in the sense of von Neumann and
a family of invertible bimodules over the regular algebra are associated. These data describe
completely, as a triangulated category, the singularity category of the artin algebra.
A criterion on the Hom-finiteness of the singularity category is given in terms of
the valued quiver of the artin algebra.
\end{abstract}

\section{Introduction}
Let $R$ be  a commutative artinian ring. All algebras, categories and functors are $R$-linear.
We recall that an $R$-linear category is \emph{Hom-finite} provided that all the Hom sets
are finitely generated $R$-modules.

Let $A$ be an artin $R$-algebra. Denote by $A\mbox{-mod}$ the category of finitely generated
left $A$-modules, and by $\mathbf{D}^b(A\mbox{-mod})$ the bounded derived category.
Following \cite{Or04}, the \emph{singularity category} $\mathbf{D}_{\rm sg}(A)$ is the quotient
triangulated category of $\mathbf{D}^b(A\mbox{-mod})$ with respect to the full subcategory
formed by perfect complexes; see also \cite{Buc, KV, Hap91, Ric, Bel2000} and \cite{Kra}. Here, we recall that a complex in  $\mathbf{D}^b(A\mbox{-mod})$ is \emph{perfect} provided that it is isomorphic to a bounded complex consisting of finitely
generated projective modules.

The singularity category  measures the homological singularity of an algebra in the sense that an algebra $A$ has finite global dimension if and only if its singularity category $\mathbf{D}_{\rm sg}(A)$ vanishes. In the meantime, the singularity
 category captures the stable homological features of an algebra (\cite{Buc}). A fundamental result
of Buchweitz and Happel states that for a Gorenstein algebra $A$, the singularity category $\mathbf{D}_{\rm sg}(A)$ is triangle
equivalent to the stable category of (maximal) Cohen-Macaulay $A$-modules (\cite{Buc, Hap91}). This implies in particular that
the singularity category of a Gorenstein algebra is Hom-finite and has Auslander-Reiten triangles. We point
out that Buchweitz and Happel's result specializes to Rickard's result (\cite{Ric}) on self-injective algebras.  However,
for non-Gorenstein algebras, not much is known about their singularity categories (\cite{Ch09}).

Our aim is to describe the singularity category of an algebra with radical square zero.
We point out that such algebras are usually non-Gorenstein (\cite{Ch11}). In what follows, we describe the
results in this paper.

We denote by $\mathbf{r}$ the Jacobson radical of $A$. The algebra $A$
is said to be with \emph{radical square zero} provided that $\mathbf{r}^2=0$. In this case,
$\mathbf{r}$ has a natural ${A/{\mathbf{r}}}$-${A/\mathbf{r}}$-bimodule structure.
Set $\mathbf{r}^{\otimes 0}={A/{\mathbf{r}}}$ and
 $\mathbf{r}^{\otimes i+1}=\mathbf{r} \otimes_{A/{\mathbf{r}}}(\mathbf{r}^{\otimes i}) $ for $i\geq 0$.
 Then there are obvious algebra homomorphisms ${\rm End}_{A/{\mathbf{r}}}(\mathbf{r}^{\otimes i})\rightarrow {\rm End}_{A/{\mathbf{r}}}(\mathbf{r}^{\otimes i+1})$ induced by $\mathbf{r}\otimes_{A/{\mathbf{r}}}-$. We denote by
 $\Gamma(A)$ the direct limit of this chain of algebra homomorphisms. It is a regular algebra (\cite{Fai73, Fai76}) in the sense of von Neumann. We call $\Gamma(A)$ the \emph{associated regular algebra} of $A$. In most
 cases, the algebra $\Gamma(A)$ is not semisimple.

 For $n\in \mathbb{Z}$ and $i\geq {\rm max}\{0, n\}$, ${\rm Hom}_{A/\mathbf{r}}(\mathbf{r}^{\otimes i}, \mathbf{r}^{\otimes i-n})$
 has a natural ${\rm End}_{A/\mathbf{r}}(\mathbf{r}^{\otimes i-n})$-${\rm End}_{A/\mathbf{r}}(\mathbf{r}^{\otimes i})$-bimodule structure. Set $K^n(A)$ to be the direct limit of the chain of maps ${\rm Hom}_{A/\mathbf{r}}(\mathbf{r}^{\otimes i}, \mathbf{r}^{\otimes i-n})\rightarrow {\rm Hom}_{A/\mathbf{r}}(\mathbf{r}^{\otimes i+1}, \mathbf{r}^{\otimes i+1-n})$, which are
 induced by $\mathbf{r}\otimes_{A/\mathbf{r}}-$. Then $K^n(A)$ is naturally a $\Gamma(A)$-$\Gamma(A)$-bimodule for each $n\in \mathbb{Z}$. Observe that $K^0(A)={_{\Gamma(A)}\Gamma(A)_{\Gamma(A)}}$ as bimodules, and that composition of maps induces $\Gamma(A)$-$\Gamma(A)$-bimodule morphisms $\phi^{n, m}\colon K^n(A)\otimes_{\Gamma(A)}K^m(A)\rightarrow K^{n+m}(A)$ for all $n, m\in \mathbb{Z}$. These bimodules $K^n(A)$
are  called the \emph{associated bimodules} of $A$.

Recall that for an algebra $\Gamma$, a $\Gamma$-$\Gamma$-bimodule $K$ is \emph{invertible} provided that the functor
$K\otimes_\Gamma -$ induces an auto-equivalence on the category of left $\Gamma$-modules.

\vskip 5pt

\noindent{\bf Theorem A.}\; \emph{Let $A$ be an artin algebra with radical square zero. Use the notation
as above. Then the associated $\Gamma(A)$-$\Gamma(A)$-bimodules $K^n(A)$ are invertible and the maps $\phi^{n, m}$
are isomorphisms of bimodules. }

\vskip 5pt

Since the algebra $\Gamma(A)$ is regular, the category ${\rm proj}\; \Gamma(A)$ of finitely generated right
projective $\Gamma(A)$-module is a semisimple abelian category. The invertible bimodule $K^1(A)$ induces an
auto-equivalence $\Sigma_A= -\otimes_{\Gamma(A)}K^1(A)\colon {\rm proj}\; \Gamma(A)\rightarrow {\rm proj}\; \Gamma(A)$. We observe
that the category ${\rm proj}\; \Gamma(A)$ has a unique triangulated structure with $\Sigma_A$ its shift functor; see Lemma \ref{lem:abelian}. This unique triangulated category is denoted by $({\rm proj}\; \Gamma(A), \Sigma_A)$.

The following result describes the singularity category of an artin algebra with radical square zero, which is
based on a result by Keller and Vossieck (\cite{KV}).

\vskip 5pt

\noindent {\bf Theorem B.}\; \emph{Let $A$ be an artin algebra with radical square zero.
Use the notation as above. Then we have a triangle equivalence
$$\mathbf{D}_{\rm sg}(A) \simeq ({\rm proj}\; \Gamma(A), \Sigma_A).$$}

We are interested in the Hom-finiteness of singularity categories. For this, we recall the notion of
valued quiver of an artin algebra $A$. Choose a complete set of
representatives of pairwise non-isomorphic simple $A$-modules
$\{S_1, S_2, \cdots, S_n\}$. Set $\Delta_i={\rm End}_A(S_i)$; they are division algebras. Observe that ${\rm Ext}^1_A(S_i,
S_j)$ has a natural $\Delta_j$-$\Delta_i$-bimodule structure. The
\emph{valued quiver} $Q_A$ of $A$ is defined as follows: its
vertex set is $\{S_1, S_2, \cdots, S_n\}$, here we identify each
simple module $S_i$ with its isoclass; there is an arrow from $S_i$ to $S_j$
whenever ${\rm Ext}_A^1(S_i, S_j)\neq 0$, in which case the arrow is
endowed with a \emph{valuation} $({\rm dim}_{\Delta_j} \; {\rm
Ext}^1_A(S_i, S_j), {\rm dim}_{{\Delta_i}^{\rm op}} \; {\rm
Ext}_A^1(S_i, S_j) )$; here ${\Delta_i}^{\rm op}$ denotes the
opposite algebra of $\Delta_i$. We say that the valuation of $Q_A$
is \emph{trivial} provided that all the valuations are $(1, 1)$. Recall that
a vertex in a valued quiver is a source (\emph{resp.} sink) provided that there is no
arrow ending (\emph{resp.} starting) at  it. For a valued quiver, to adjoin a (new)
source (\emph{resp.} sink) is to add a vertex together with some valued
arrows starting (\emph{resp.} ending) at this vertex.  For details, we refer to
\cite[III.1]{ARS}.

The following result characterizes the Hom-finiteness of the singularity category in terms of
valued quivers.
\vskip 5pt

\noindent{\bf Theorem C.}\; \emph{Let $A$ be an artin algebra with radical square zero. Then the following statements
are equivalent:
\begin{enumerate}
\item the singularity category $\mathbf{D}_{\rm sg}(A)$ is Hom-finite;
\item the associated regular algebra $\Gamma(A)$ is semisimple;
\item the valued quiver $Q_A$ is obtained from a disjoint union of oriented cycles with the trivial valuation
by repeatedly adjoining sources or sinks.
\end{enumerate}}

\vskip 5pt

The paper is structured as follows. In Section 2, we collect some basic facts on  singularity categories and recall a
basic result of Keller and Vossieck. We prove Theorem A and B in Section 3, where an explicit example is presented. In Section 4, we prove that one-point extensions and coextensions of algebras preserve their singularity categories. We introduce the notion
of cyclicization of an algebra, which is used in the proof of Theorem C in Section 5.

For artin algebras, we refer to \cite{ARS}. For triangulated categories, we refer to \cite{Har66} and \cite{Hap88}.

\section{Preliminaries}

In this section, we collect some facts on singularity categories of artin algebras. We recall a basic result
due to Keller and Vossieck (\cite{KV}), which is applied to $\Omega^\infty$-finite algebras.

Let $A$ be an artin algebra over a commutative artinian ring $R$. Recall that
$A\mbox{-mod}$ denotes the category of finitely generated left $A$-modules. We denote by
$A\mbox{-proj}$ the full subcategory formed by projective modules, and by $A\mbox{-\underline{mod}}$
the stable category of $A\mbox{-mod}$ modulo projective modules (\cite[p.104]{ARS}). The morphism space $\underline{\rm Hom}_A(M, N)$ of two modules $M$ and $N$ in $A\mbox{-\underline{mod}}$ is defined to be ${\rm Hom}_A(M, N)/\mathbf{p}(M, N)$, where $\mathbf{p}(M, N)$ denotes the $R$-submodule formed by morphisms that factor through projective modules.

Recall that for an $A$-module $M$, its syzygy $\Omega(M)$ is the kernel of its projective cover $P\rightarrow M$. This gives rise to the \emph{syzygy functor} $\Omega\colon A\mbox{-\underline{mod}}\rightarrow A\mbox{-\underline{mod}}$ (\cite[p.124]{ARS}). Set $\Omega^0(M)=M$ and $\Omega^{i+1}(M)=\Omega^i(\Omega(M))$ for $i\geq 0$. Denote by $\Omega^i(A\mbox{-mod})$ the full
subcategory of $A\mbox{-mod}$ formed by modules of the form $P\oplus\Omega^i(M)$ for some module $M$ and projective module $P$. Then an $A$-module $X$ belongs to $\Omega^i(A\mbox{-mod})$ if and only if there is an exact sequence $0\rightarrow X\rightarrow P^{1-i}\rightarrow \cdots \rightarrow P^{-1} \rightarrow  P^0$ with each $P^j$ projective.

Recall that $\mathbf{D}^b(A\mbox{-mod})$ denotes the bounded derived category of $A\mbox{-mod}$, whose shift functor is denoted by $[1]$. For $n\in \mathbb{Z}$, $[n]$ denotes the $n$-th power of $[1]$.   The module category $A\mbox{-mod}$ is viewed as a full subcategory of  $\mathbf{D}^b(A\mbox{-mod})$ by identifying an $A$-module with the corresponding stalk complex concentrated at degree zero (\cite[Proposition I.4.3]{Har66}). Recall that a complex in $\mathbf{D}^b(A\mbox{-mod})$  is \emph{perfect} provided that it is isomorphic to a bounded complex consisting of projective modules; these complexes form a full triangulated subcategory ${\rm perf}(A)$. Recall that, via an obvious functor,  ${\rm perf}(A)$ is triangle equivalent to the bounded homotopy category $\mathbf{K}^b(A\mbox{-proj})$; compare \cite[1.1-1.2]{Buc}.

Following \cite{Or04}, we call the quotient triangulated category $$\mathbf{D}_{\rm sg}(A)=\mathbf{D}^b(A\mbox{-mod})/{{\rm perf}(A)}$$ the \emph{singularity category} of $A$. Denote by $q\colon \mathbf{D}^b(A\mbox{-mod})\rightarrow \mathbf{D}_{\rm sg}(A)$ the quotient functor.

The following two results are known; compare \cite[Lemma 1.11]{Or04} and \cite[Lemma 2.2.2]{Buc}.

\begin{lem}\label{lem:stalk}
Let $X^\bullet$ be a complex in $\mathbf{D}_{\rm sg}(A)$ and $n_0>0$. Then for any $n$ large enough, there exists
a module $M$ in $\Omega^{n_0}(A\mbox{-{\rm mod}})$ such that $X^\bullet\simeq q(M)[n]$.
\end{lem}

\begin{proof}
Take a quasi-isomorphism $P^\bullet \rightarrow X^\bullet$ with $P^\bullet$ a bounded above complex of projective modules (\cite[Lemma I.4.6]{Har66}).
Take $n\geq n_0$ such that $H^i(X^\bullet)=0$ for all $i\leq n_0-n$. Consider the good truncation
$\sigma^{\geq -n}P^\bullet=\cdots \rightarrow  0\rightarrow M\rightarrow P^{1-n}\rightarrow P^{2-n}\rightarrow \cdots $ of $P^\bullet$, which is quasi-isomorphic to $P^\bullet$. Then the cone of the obvious chain map $\sigma^{\geq -n}P^\bullet\rightarrow M[n]$ is perfect, which becomes an isomorphism in $\mathbf{D}_{\rm sg}(A)$.  This shows that $X^\bullet\simeq q(M)[n]$.
We observe that $M$ lies in $\Omega^{n_0}(A\mbox{-{\rm mod}})$.
\end{proof}

\begin{lem}\label{lem:exactsequence}
Let $0\rightarrow M\rightarrow P^{1-n}\rightarrow \cdots \rightarrow P^0\rightarrow N\rightarrow 0$ be an exact sequence
with each $P^i$ projective. Then we have an isomorphism $q(N)\simeq q(M)[n]$ in $\mathbf{D}_{\rm sg}(A)$. In particular, for an
$A$-module $M$, we have a natural isomorphism $q(\Omega^n(M))\simeq q(M)[-n]$.
\end{lem}

\begin{proof}
The stalk complex $N$ is quasi-isomorphic to $\cdots \rightarrow 0\rightarrow M\rightarrow P^{1-n}\rightarrow \cdots \rightarrow P^0\rightarrow 0\rightarrow \cdots$. This gives rise to a morphism $N\rightarrow M[n]$ in $\mathbf{D}^b(A\mbox{-mod})$, whose cone is perfect.
Then this morphism becomes an isomorphism in  $\mathbf{D}_{\rm sg}(A)$.
\end{proof}

Consider the composite $q'\colon A\mbox{-mod}\hookrightarrow \mathbf{D}^b(A\mbox{-mod})\stackrel{q}\rightarrow \mathbf{D}_{\rm sg}(A)$; it vanishes on projective modules. Then it induces uniquely a functor $A\mbox{-\underline{mod}}\rightarrow \mathbf{D}_{\rm sg}(A)$, which is still denoted by $q'$. Then Lemma \ref{lem:exactsequence} yields, for each $n\geq 0$,  the following
commutative diagram
\[\xymatrix{
A\mbox{-\underline{mod}}\ar[d]^{q'} \ar[r]^-{\Omega^n} & A\mbox{-\underline{mod}}\ar[d]^-{q'}\\
\mathbf{D}_{\rm sg}(A) \ar[r]^-{[-n]} & \mathbf{D}_{\rm sg}(A).
}\]
We refer to \cite[Lemma 2.2.2]{Buc} for a similar statement.

 The functor $q'$ induces a natural map $$\Phi^0\colon{\rm \underline{Hom}}_A(M, N)\rightarrow {\rm Hom}_{\mathbf{D}_{\rm sg}(A)}(q(M), q(N))$$ for any modules $M, N$.  Let $n\geq 1$. Lemma \ref{lem:exactsequence} yields a natural isomorphism $\theta_M\colon q(M)\stackrel{\sim}\longrightarrow q(\Omega^n(M))[n]$. Then we have a map $$\Phi^n\colon {\rm \underline{Hom}}_A(\Omega^n(M), \Omega^n(N))\rightarrow {\rm Hom}_{\mathbf{D}_{\rm sg}(A)}(q(M), q(N))$$
given by $\Phi^n(f)=(\theta^n_N)^{-1}\circ (\Phi^0(f)[n])\circ \theta^n_M$.

Consider the chain of maps
${\rm \underline{Hom}}_A(\Omega^n(M), \Omega^n(N))\rightarrow {\rm \underline{Hom}}_A(\Omega^{n+1}(M), \Omega^{n+1}(N))$ induced by the syzygy functor. It is routine to verify that $\Phi^n$ are compatible with this chain of maps. Then we have an induced map
$$\Phi\colon  \varinjlim {\rm \underline{Hom}}_A(\Omega^n(M), \Omega^n(N)) \longrightarrow {\rm Hom}_{\mathbf{D}_{\rm sg}(A)}(q(M), q(N)).$$

We recall the following basic result.

\begin{prop}{\rm (Keller-Vossieck)}\label{prop:singlimit} Let $M, N$ be $A$-modules as above. Then  the map $\Phi$ is an isomorphism.
\end{prop}

\begin{proof}
The statement follows from \cite[Exemple 2.3]{KV}. We refer to \cite[Corollary 3.9(1)]{Bel2000} for a detailed proof.
\end{proof}

Recall that an additive category $\mathcal{A}$ is \emph{idempotent split} provided that
each idempotent $e\colon X\rightarrow X$  splits, that is, it admits a factorization $X\stackrel{u}\rightarrow Y\stackrel{v}\rightarrow X$
with $u\circ v={\rm Id}_Y$. For example, a Krull-Schmidt category is idempotent split (\cite[Appendix A]{CYZ}). In particular, for an artin
algebra $A$, the stable category $A\mbox{-\underline{mod}}$ is idempotent split.

\begin{cor}\label{cor:idempotent}
The singularity category $\mathbf{D}_{\rm sg}(A)$ of an artin algebra $A$ is idempotent split.
\end{cor}

\begin{proof}
 By Lemma \ref{lem:stalk} it suffices to show that for each module $M$, an idempotent $e\colon q(M)\rightarrow q(M)$ splits.
 The above proposition implies that for a large $n$, there is an idempotent $e^n\colon \Omega^n(M)\rightarrow \Omega^n(M)$ in $A\mbox{-\underline{mod}}$ which is mapped by $\Phi$ to $e$. The idempotent $e^n$ splits as $\Omega^n(M)\stackrel{u} \rightarrow Y\stackrel{v}\rightarrow \Omega^n(M)$ with $u\circ v={\rm Id}_Y$ in $A\mbox{-\underline{mod}}$. Then the idempotent $e$
 factors as $q(M)\stackrel{(q(u)[n])\circ \theta_M^n}\longrightarrow q(Y)[n]\stackrel{(\theta_M^n)^{-1}\circ (q(v)[n])}\longrightarrow q(M)$.
\end{proof}

Let $\mathcal{A}$ be an additive category. For a subcategory $\mathcal{C}$, denote by ${\rm add}\; \mathcal{C}$
the full subcategory of $\mathcal{A}$ formed by direct summands of finite direct sums of objects in $\mathcal{C}$.
For any algebra $\Gamma$, denote by $\mbox{proj}\; \Gamma$ the category of finitely generated right projective $\Gamma$-modules. We observe
that $\mbox{proj}\; \Gamma={\rm add}\; \Gamma_\Gamma$.

An artin algebra $A$ is called \emph{$\Omega^{\infty}$-finite} provided that there exists a module $E$
and $n\geq 0$ such that $\Omega^n(A\mbox{-mod})\subseteq {\rm add}\;(A\oplus E)$. In this case, we call
$E$ an \emph{$\Omega^\infty$-generator} of $A$.

\begin{prop}\label{prop:sing}
Let $A$ be an  $\Omega^{\infty}$-finite algebra with an $\Omega^\infty$-generator $E$. Then we have
$\mathbf{D}_{\rm sg}(A)={\rm add}\; q(E)$. Consequently, we have an equivalence of categories
$$\mathbf{D}_{\rm sg}(A) \simeq {\rm proj}\; {\rm End}_{\mathbf{D}_{\rm sg}(A)}(q(E)),$$
which sends $q(E)$ to ${\rm End}_{\mathbf{D}_{\rm sg}(A)}(q(E))$.
\end{prop}

\begin{proof}
Observe that $\Omega^{n+1}(A\mbox{-mod})\subseteq \Omega^n(A\mbox{-mod})$. Then we may assume that
${\rm add}\; (A\oplus E)\supseteq {\rm add}\; \Omega^{n_0}(A\mbox{-mod})={\rm add}\; \Omega^{n_0+1}(A\mbox{-mod})=\cdots $
for  $n_0$ large enough.

For the first statement, it suffices to show that each object $X^\bullet$ in $\mathbf{D}_{\rm sg}(A)$ belongs to
${\rm add}\; q(E)$. By Lemma \ref{lem:stalk}, $X^\bullet\simeq q(M)[n_1]$ for a module $M\in \Omega^{n_0}(A\mbox{-mod})$ and
$n_1>0$.  Since ${\rm add}\; \Omega^{n_0}(A\mbox{-mod})={\rm add}\; \Omega^{n_0+n_1}(A\mbox{-mod})$, we may assume that
$M\oplus N\in \Omega^{n_0+n_1}(A\mbox{-mod})$ for some module $N$. Take an exact sequence $0\rightarrow M\oplus N\rightarrow P^{1-n_1}\rightarrow
\cdots \rightarrow P^0\rightarrow L\rightarrow 0$ with each $P^i$ projective and $L\in \Omega^{n_0}(A\mbox{-mod})$. By Lemma
\ref{lem:exactsequence}, $q(L)\simeq q(M\oplus N)[n_1]$ and then $X^\bullet$ is a direct summand of $q(L)$. Observing that
$L\in {\rm add}\; (A\oplus E)$, we are done with the first statement.

The second statement follows from the projectivization; see \cite[Proposition II.2.1]{ARS}. The functor is given by
${\rm Hom}_{\mathbf{D}_{\rm sg}(A)}(q(E), -)$. We point out that Corollary \ref{cor:idempotent} is needed here.
\end{proof}

\section{Algebras with radical square zero}

In this section, we study the singularity category of an algebra with radical square zero, and prove
Theorem A and B. An explicit example is given at the end.

Let $A$ be an artin algebra. Denote by $\mathbf{r}$ the Jacobson radical of $A$.
The algebra $A$ is said to be with \emph{radical square zero} provided that
$\mathbf{r}^2=0$. In this case, $\mathbf{r}$ has an
${A/{\mathbf{r}}}\mbox{-}{A/\mathbf{r}}\mbox{-}$bimodule structure, which is induced from
the multiplication of $A$.

Denote by $A\mbox{-\underline{ssmod}}$ the full subcategory of $A\mbox{-\underline{mod}}$ formed
by semisimple modules. We observe that $\mathbf{r}\otimes_{A/\mathbf{r}}S=0$ for a simple projective
module $S$. Then the functor $\mathbf{r}\otimes_{A/\mathbf{r}}-\colon A\mbox{-\underline{ssmod}}\rightarrow A\mbox{-\underline{ssmod}}$ is well defined. We observe that the syzygy functor $\Omega$ sends semisimple modules to
semisimple modules, and then we have the restricted functor $\Omega\colon A\mbox{-\underline{ssmod}}\rightarrow A\mbox{-\underline{ssmod}}$.

The following result is implicitly contained in the proof of \cite[Lemma X.2.1]{ARS}.

\begin{lem}\label{lem:omega}
There is a natural isomorphism $\Omega\simeq \mathbf{r}\otimes_{A/\mathbf{r}}-$ of
functors on $A\mbox{-\underline{\rm ssmod}}$.
\end{lem}

\begin{proof}
Let $X$ be a semisimple module with a projective cover $P\rightarrow X$. Tensoring $P$
with the natural exact sequence of $A$-$A$-bimodules $0\rightarrow \mathbf{r}\rightarrow A\rightarrow {A/\mathbf{r}}\rightarrow 0$ yields
$\Omega(X)\simeq \mathbf{r} \otimes_A P$. Using isomorphisms $\mathbf{r}\otimes_A P\simeq
\mathbf{r}\otimes_{A/\mathbf{r}} {P/\mathbf{r}P}$ and ${P/\mathbf{r}P}\simeq X$, we get an isomorphism
$\Omega(X)\simeq \mathbf{r}\otimes_{A/\mathbf{r}}X$. It is routine to verify that this isomorphism
is natural in $X$.
\end{proof}

Recall that an algebra $\Gamma$ is \emph{regular} in the sense of von Neumann provided that for each
element $a$ there exists $a'$ such that $aa'a=a$. For example, a semisimple algebra is regular.
Then a direct limit of semisimple algebras is regular.  For details, we refer to \cite[Theorem and Definition 11.24]{Fai73}.

Recall that for an artin algebra $A$ with radical square zero, there is a chain of algebra homomorphisms
${\rm End}_{A/{\mathbf{r}}}(\mathbf{r}^{\otimes i})\rightarrow {\rm End}_{A/{\mathbf{r}}}(\mathbf{r}^{\otimes i+1})$ induced by $\mathbf{r}\otimes_{A/{\mathbf{r}}}-$. Here, $\mathbf{r}^{\otimes 0}=A/\mathbf{r}$ and $\mathbf{r}^{\otimes i+1}=\mathbf{r} \otimes_{A/\mathbf{r}}  (\mathbf{r}^{\otimes i})$. We set $\Gamma(A)$ to be the direct limit of this chain. Since each algebra ${\rm End}_{A/{\mathbf{r}}}(\mathbf{r}^{\otimes i})$ is semisimple, the algebra $\Gamma(A)$ is regular. It is called the
\emph{associated regular algebra} of $A$. We refer to \cite[19.26B, Example]{Fai76} for a related construction.

We recall the \emph{associated $\Gamma(A)$-$\Gamma(A)$-bimodules} $K^n(A)$ of $A$, $n\in \mathbb{Z}$.
 For $i\geq {\rm max}\{0, n\}$, ${\rm Hom}_{A/\mathbf{r}}(\mathbf{r}^{\otimes i}, \mathbf{r}^{\otimes i-n})$
 has a natural ${\rm End}_{A/\mathbf{r}}(\mathbf{r}^{\otimes i-n})$-${\rm End}_{A/\mathbf{r}}(\mathbf{r}^{\otimes i})$-bimodule structure. Consider a chain of maps ${\rm Hom}_{A/\mathbf{r}}(\mathbf{r}^{\otimes i}, \mathbf{r}^{\otimes i-n})\rightarrow {\rm Hom}_{A/\mathbf{r}}(\mathbf{r}^{\otimes i+1}, \mathbf{r}^{\otimes i+1-n})$, which are
 induced by $\mathbf{r}\otimes_{A/\mathbf{r}}-$, and define $K^n(A)$ to be its direct limit. Then $K^n(A)$ is naturally a $\Gamma(A)$-$\Gamma(A)$-bimodule for each $n\in \mathbb{Z}$. Observe that $K^0(A)={_{\Gamma(A)}\Gamma(A)_{\Gamma(A)}}$ as $\Gamma(A)$-$\Gamma(A)$-bimodules.

\begin{prop}\label{prop:K^n(A)}
Let $A$ be an artin algebra with radical square zero. Then there is a natural isomorphism
$$K^n(A)\simeq {\rm Hom}_{{\mathbf{D}_{\rm sg}(A)}}(q(A/\mathbf{r}), q(A/\mathbf{r})[n])$$
for each $n\in \mathbb{Z}$.
\end{prop}

\begin{proof}
Consider the case $n\leq 0$ first. In this case, by Lemmas \ref{lem:exactsequence} and  \ref{lem:omega} we have $q(A/\mathbf{r})[n]\simeq q(\Omega^{-n}(A/\mathbf{r}))\simeq q(\mathbf{r}^{\otimes -n})$. Then Proposition
\ref{prop:singlimit} yields an isomorphism ${\rm Hom}_{{\mathbf{D}_{\rm sg}(A)}}(q(A/\mathbf{r}), q(A/\mathbf{r})[n])\simeq \varinjlim {\rm \underline{Hom}}_A(\Omega^i(A/\mathbf{r}), \Omega^i(\mathbf{r}^{\otimes -n}))$. By Lemma \ref{lem:omega}
again we have $\Omega^i(A/\mathbf{r})\simeq \mathbf{r}^{\otimes i}$ and $\Omega^i(\mathbf{r}^{\otimes -n})=\mathbf{r}^{\otimes i-n}$. Then we have a surjective map $\psi\colon K^n(A)\rightarrow {\rm Hom}_{{\mathbf{D}_{\rm sg}(A)}}(q(A/\mathbf{r}), q(A/\mathbf{r})[n])$. On the other hand, every morphism $f\colon \mathbf{r}^{\otimes i}\rightarrow \mathbf{r}^{\otimes i-n}$
 that is zero in $A\mbox{-\underline{mod}}$
 necessarily factors through a semisimple projective module. However, the functor $\mathbf{r}\otimes_{A/\mathbf{r}}-$ vanishes
 on semisimple projective modules. Then $\mathbf{r}\otimes_{A/\mathbf{r}} f$ is zero. This forces that $\psi$ is injective.
 We are done in this case.

 For the case $n>0$, we observe that ${\rm Hom}_{{\mathbf{D}_{\rm sg}(A)}}(q(A/\mathbf{r}), q(A/\mathbf{r})[n])$ is
 isomorphic to ${\rm Hom}_{{\mathbf{D}_{\rm sg}(A)}}(q(A/\mathbf{r})[-n], q(A/\mathbf{r}))$, and by the same argument
 as above, it is isomorphic to $\varinjlim {\rm \underline{Hom}}_{A/\mathbf{r}}(\mathbf{r}^{\otimes i+n}, \mathbf{r}^{\otimes i})$.
 Then we get a surjective map $K^n(A)\rightarrow {\rm Hom}_{{\mathbf{D}_{\rm sg}(A)}}(q(A/\mathbf{r}), q(A/\mathbf{r})[n])$.
Similarly as above, we have that this map is injective.
\end{proof}

\begin{rem}\label{rem:bimodule}
In the case $n=0$, the above isomorphism is an isomorphism $\Gamma(A)\simeq {\rm End}_{\mathbf{D}_{\rm sg}(A)}(q(A/\mathbf{r}))$ of algebras. Then for an arbitrary $n$, the above isomorphism becomes an isomorphism of $\Gamma(A)$-$\Gamma(A)$-bimodules.
\end{rem}

Recall that an abelian category $\mathcal{A}$ is \emph{semisimple} provided that each short exact sequence
splits. For example, for a regular algebra $\Gamma$, the category ${\rm proj}\; \Gamma$ of finitely generated right projective $\Gamma$-modules is a semisimple abelian category. Here, we use the fact that all finitely presented $\Gamma$-modules
are projective; see \cite[Theorem and Definition 11.24(a)]{Fai73}.

The following observation is well known.

\begin{lem}\label{lem:abelian}
Let $\mathcal{A}$ be a semisimple abelian category, and let $\Sigma$ be an auto-equivalence on $\mathcal{A}$. Then there is
a unique triangulated structure on $\mathcal{A}$ with $\Sigma$ the shift functor.
\end{lem}

The obtained triangulated category in this lemma will be denoted by $(\mathcal{A}, \Sigma)$.

\begin{proof}
We use the fact that each morphism in $\mathcal{A}$ is isomorphic to a direct sum of morphisms
of the forms $K\rightarrow 0$, $I\stackrel{{\rm Id}_I}\rightarrow I$ and $0\rightarrow C$. Then
all possible triangles are a direct sum of the following trivial triangles $K\rightarrow 0\rightarrow \Sigma(K)\stackrel{{\rm Id}_{\Sigma(K)}}\rightarrow \Sigma(K)$, $I\stackrel{{\rm Id}_I}\rightarrow I \rightarrow 0\rightarrow \Sigma(I)$ and
$0\rightarrow C\stackrel{{\rm Id}_C}\rightarrow C\rightarrow \Sigma(0)$.
\end{proof}

\begin{prop}\label{prop:sg}
Let $A$ be an artin algebra with radical square zero and let $\Gamma(A)$ be its associated
regular algebra. Then there is a triangle equivalence
$$ \Psi\colon \mathbf{D}_{\rm sg}(A)\simeq ({\rm proj}\; \Gamma(A), \Sigma)$$
for some auto-equivalence $\Sigma$ on ${\rm proj}\; \Gamma(A)$, which sends $q(A/\mathbf{r})$ to $\Gamma(A)$.
\end{prop}

\begin{proof}
We observe that for any $A$-module $M$, its syzygy $\Omega(M)$ is semisimple. Hence we have
$\Omega^{1}(A\mbox{-mod})\subseteq {\rm add}\; (A\oplus A/\mathbf{r})$. We apply Proposition \ref{prop:sing}
to obtain an equivalence of categories $\mathbf{D}_{\rm sg}(A) \simeq {\rm proj}\; {\rm End}_{\mathbf{D}_{\rm sg}(A)}(q(A/\mathbf{r}))$. By Proposition \ref{prop:K^n(A)} this yields an equivalence of categories $\mathbf{D}_{\rm sg}(A)\simeq {\rm proj}\; \Gamma(A)$.

 By transport of structures, the shift functor $[1]$ on $\mathbf{D}_{\rm sg}(A)$ corresponds to an auto-equivalence
 $\Sigma$ on ${\rm proj}\; \Gamma(A)$, and then ${\rm proj}\; \Gamma(A)$ becomes a triangulated category. However, by Lemma
\ref{lem:abelian} the semisimple abelian category ${\rm proj}\; \Gamma(A)$ has a unique triangulated structure with $\Sigma$
the shift functor. Then this structure necessarily coincides with the transported one. Then we are done.
\end{proof}

We are interested in the auto-equivalence $\Sigma$ above. The following result characterizes
it using the bimodules $K^n(A)$.

\begin{lem}\label{lem:Sigma}
Use the notation as above. Then for each $n\in \mathbb{Z}$, the auto-equivalence $\Sigma^n$ is isomorphic
to $-\otimes_{\Gamma(A)}K^n(A)\colon {\rm proj}\; \Gamma(A)\rightarrow {\rm proj}\; \Gamma(A)$.
\end{lem}

\begin{proof}
Recall that the above equivalence $\Psi$ is given by ${\rm Hom}_{\mathbf{D}_{\rm sg}(A)}(q(A/\mathbf{r}),-)$, which
sends $q(A/\mathbf{r})$ to $\Gamma(A)$. The auto-equivalence $\Sigma^n$ corresponds, via $\Psi$, to $[n]$ on $\mathbf{D}_{\rm sg}(A)$. Then by Proposition \ref{prop:K^n(A)} we have an isomorphism  $\phi\colon K^n(A)\stackrel{\sim}\longrightarrow \Sigma^n(\Gamma(A))$ of
right $\Gamma(A)$-modules. Recall that $\Sigma^n(\Gamma(A))$ has a natural $\Gamma(A)$-$\Gamma(A)$-bimodule structure such
that $\Sigma^n$ is isomorphic to $-\otimes_{\Gamma(A)} \Sigma^n(\Gamma(A))$. Thanks to  Remark \ref{rem:bimodule}, the isomorphism $\phi$ is an isomorphism of bimodules. This proves the lemma.
\end{proof}

Recall that for an algebra $\Gamma$, a $\Gamma$-$\Gamma$-bimodule $K$ is \emph{invertible} provided that the functor
$-\otimes_\Gamma K$ induces an auto-equivalence on the category of right $\Gamma$-modules. For details, we refer to
\cite[Definition and Proposition 12.13]{Fai73}.

We recall that for an artin algebra $A$ with radical square zero, the associated $\Gamma(A)$-$\Gamma(A)$-bimodules
$K^n(A)$ are defined to be $\varinjlim {\rm Hom}_{A/\mathbf{r}}(\mathbf{r}^{\otimes i}, \mathbf{r}^{\otimes i-n})$, where $i\geq {\rm max}\{0, n\}$. Then composition of maps between the $A/\mathbf{r}$-modules $\mathbf{r}^{\otimes j}$ yields morphisms
$$\phi^{n, m}\colon K^n(A)\otimes_{\Gamma(A)} K^m(A)\longrightarrow K^{n+m}(A)$$
of $\Gamma(A)$-$\Gamma(A)$-bimodules, for all $n, m\in \mathbb{Z}$. More precisely, let $f\in K^n(A)$ and
$g\in K^m(A)$ be represented by $f'\colon \mathbf{r}^{\otimes j-m} \rightarrow  \mathbf{r}^{\otimes j-m-n}$ and
$g'\colon \mathbf{r}^{\otimes j}\rightarrow \mathbf{r}^{\otimes j-m}$ for some large $j$, respectively.  Then $\phi^{n, m}(f\otimes g)$ is represented by the composite $f'\circ g'$.

The following result is Theorem A.

\begin{thm}\label{thm:A}
Let $A$ be an artin algebra with radical square zero. Use the notation as above. Then for all $n,m\in \mathbb{Z}$,
the $\Gamma(A)$-$\Gamma(A)$-bimodules $K^n(A)$ are invertible and the morphisms $\phi^{n, m}$ are isomorphisms.
\end{thm}

\begin{proof}
By Lemma \ref{lem:Sigma} the functor $-\otimes_{\Gamma(A)}K^n(A)\colon {\rm proj}\; \Gamma(A)\rightarrow {\rm proj}\; \Gamma(A)$
is an auto-equivalence for each $n$. This functor extends naturally to an auto-equivalence on the category of all right $\Gamma(A)$-modules.
Then $K^n(A)$ is an invertible bimodule. The second statement follows from  Lemma \ref{lem:Sigma} and the fact that $\Sigma^{m}\Sigma^n$ is isomorphic to $\Sigma^{n+m}$. Here, we use \cite[Proposition 12.9]{Fai73} implicitly.
\end{proof}

We now have Theorem B. Denote the functor $-\otimes_{\Gamma(A)}K^1(A)\colon {\rm proj}\; \Gamma(A)\rightarrow {\rm proj}\; \Gamma(A)$ by $\Sigma_A$.

\begin{thm}\label{thm:B}
Let $A$ be an artin algebra with radical square zero.
Use the notation as above. Then we have a triangle equivalence
$$\mathbf{D}_{\rm sg}(A) \simeq ({\rm proj}\; \Gamma(A), \Sigma_A),$$
which sends $q(A/\mathbf{r})$ to $\Gamma(A)$.
\end{thm}

\begin{proof}
It follows from Proposition \ref{prop:sg} and Lemma \ref{lem:Sigma}.
\end{proof}

Let $A$ be an artin algebra with radical square zero. For each $n\geq 1$, we consider the artin algebra $G^n=A/\mathbf{r}\oplus\mathbf{r}^{\otimes n}$,
which is the \emph{trivial extension} of the $A/\mathbf{r}$-$A/\mathbf{r}$-bimodule $\mathbf{r}^{\otimes n}$ (\cite[p.78]{ARS}).
All these algebras $G^n$ have radical square zero.

The following observation seems to be of independent  interest.

\begin{prop}\label{prop:G^n}
Use the notation as above. Then for each $n\geq 1$, we have a triangle equivalence
$$\mathbf{D}_{\rm sg}(G^n) \simeq ({\rm proj}\; \Gamma(A), \Sigma^n_A).$$
In particular, we have a triangle equivalence $\mathbf{D}_{\rm sg}(A)\simeq \mathbf{D}_{\rm sg}(G^1)$.
\end{prop}

\begin{proof}
Write $G^n=A'$. Then from the very definition, we have a natural identification $\Gamma(A')=\Gamma(A)$. Moreover,
the $\Gamma(A')$-$\Gamma(A')$-bimodule $K^1(A')$ corresponds to the $\Gamma(A)$-$\Gamma(A)$-bimodule $K^n(A)$.
Then by Lemma \ref{lem:Sigma} $\Sigma_{A'}$ corresponds to $\Sigma^n_A$. Then the result follows from Theorem \ref{thm:B} immediately.
\end{proof}

\begin{rem}
We point out that for $n\geq 2$, $\mathbf{D}_{\rm sg}(G^n)$ might not be triangle equivalent
to $\mathbf{D}_{\rm sg}(A)$, although the underlying categories are equivalent.
\end{rem}

We conclude this section with an example.

\begin{exm}
Let $k$ be a field and let $n\geq 2$. Consider the algebra $A=k[x_1, x_2, \cdots, x_n]/(x_i x_j, 1\leq i, j\leq n)$, which
is with radical square zero. We identify $A/\mathbf{r}$ with $k$, and $\mathbf{r}$ with the $n$-dimensional $k$-space $V=kx_1\oplus kx_2\oplus \cdots \oplus kx_n$. Consequently, for each $i\geq 0$, the algebra ${\rm End}_{A/\mathbf{r}}(\mathbf{r}^{\otimes i})$ is isomorphic to ${\rm End}_k(V^{\otimes i})$, which is identified
with the $n^i\times n^i$ total matrix algebra $M_{n^i}(k)$. Then the associated regular algebra $\Gamma(A)$ is isomorphic to the direct limit
of the following chain of algebra embeddings
$$k\longrightarrow M_n(k)\longrightarrow M_{n^2}(k)\longrightarrow M_{n^3}(k)\longrightarrow \cdots$$
Here, for each algebra $B$, $B\rightarrow M_n(B)$ is the algebra embedding sending $b$ to $bI_{n}$ with $I_n$
the $n\times n$ identity matrix.

We observe that $\Gamma(A)$ is a simple algebra. We point out that this construction is
classical; see \cite[19.26 B, Example]{Fai76}.

Let $1\leq r, s\leq n$. Define $E_{rs}\colon V\rightarrow V$  to be the linear map such that  $E_{rs}(x_i)=\delta_{i,s}x_r$, where
$\delta$ is the Kronecker symbol. Consider,  for all $i\geq 0$, the linear maps $-\otimes_k E_{rs}\colon {\rm End}_k(V^{\otimes i})\rightarrow {\rm End}_k(V^{\otimes i+1})$. Taking the limit, we have the induced linear map $-\otimes_k E_{rs}\colon \Gamma(A)\rightarrow \Gamma(A)$ for each pair of $r, s$. Then we have an isomorphism $\sigma\colon M_n(A)\rightarrow A$ of algebras, which sends an $n\times n$ matrix $(a_{ij})$ to $\sum_{1\leq i, j\leq n}a_{ij}\otimes_k E_{ij}$.

  The associated $\Gamma(A)$-$\Gamma(A)$-bimodule $K^1(A)$ is described as follows. As a $k$-space, $K^1(A)=\Gamma(A)\oplus \Gamma(A)\oplus \cdots \Gamma(A)$ with $n$ copies of $\Gamma(A)$. The left action is given by $a(a_1, a_2, \cdots, a_n)=(aa_1, aa_2, \cdots, aa_n)$, while the right action is given by $(a_1, a_2, \cdots, a_n)a=(a_1, a_2, \cdots, a_n)\sigma^{-1}(a)$.

   We remark that the regular algebra $\Gamma(A)$ is related to a quotient abelian category studied in \cite{Smi}, which might  relate to the singularity category $\mathbf{D}_{\rm sg}(A)$ via a version of Koszul duality. 
\end{exm}

\section{One-point (co)extensions and cyclicizations of algebras}

In this section, we prove that one-point extensions and coextensions of algebras preserve their singularity categories.
We then introduce the notion of cyclicization of an algebra, which is a  repeated operation to remove sources and sinks
on the valued quiver. The obtained result will be used in the proof of Theorem C.

Let $A$ be an artin algebra. Let $D$ be a simple artin algebra and let $_AM_D$ be an $A$-$D$-bimodule, on
which $R$ acts centrally. The \emph{one-point extension} of $A$ by $M$ is the upper triangular
 matrix algebra $A[M]=\begin{pmatrix}A & M \\ 0 & D\end{pmatrix}$. A left $A[M]$-module is denoted by
 a column vector $\begin{pmatrix} X \\ V\end{pmatrix}_\phi$, where $X$ and $V$ are a  left $A$-module and $D$-module,
 respectively, and that $\phi\colon M\otimes_DV\rightarrow X$ is a morphism of $A$-modules. We sometimes
 suppress the morphism $\phi$, when it is clearly understood. For details, we refer to \cite[III.2]{ARS}.

 Recall  from \cite[III.1]{ARS} the notion of \emph{valued quiver} $Q_A$ for an artin algebra $A$. We observe that for the unique simple $D$-module $S$, the
 corresponding $A[M]$-module $\begin{pmatrix} 0\\S \end{pmatrix}$ is simple injective, which corresponds
 to a source in the valued quiver $Q_{A[M]}$ of the one-point extension $A[M]$. Indeed, this valued quiver is obtained from $Q_A$ by adding this
 source together with some valued arrows starting at it.

 One-point extensions of algebras preserve singularity categories. Observe the natural exact embedding $i\colon A\mbox{-mod}\rightarrow A[M]\mbox{-mod}$, which sends
 $_AX$ to $i(X)=\begin{pmatrix}X\\0 \end{pmatrix}$.

 \begin{prop}\label{prop:ope}
 Let $A[M]$ be the one-point extension of $A$ as above. Then the exact embedding $i\colon A\mbox{-{\rm mod}}\rightarrow A[M]\mbox{-{\rm mod}}$ induces a triangle equivalence
 $$\mathbf{D}_{\rm sg}(A)\simeq \mathbf{D}_{\rm sg}(A[M]).$$
 \end{prop}

\begin{proof}
The exact functor $i$ extends naturally to a triangle functor $i_*\colon \mathbf{D}^b(A\mbox{-mod})\rightarrow  \mathbf{D}^b(A[M]\mbox{-mod})$. We observe that $i(A)$ is projective, and then $i_*$ sends perfect complexes to perfect complexes. Then it induces a triangle functor $\bar{i_*}\colon \mathbf{D}_{\rm sg}(A)\rightarrow  \mathbf{D}_{\rm sg}(A[M])$. We claim that $\bar{i_*}$ is an equivalence.

For the claim, recall that the functor $i$ admits a right adjoint $j\colon A[M]\mbox{-mod}\rightarrow A\mbox{-mod}$
which sends $\begin{pmatrix} X \\ V\end{pmatrix}_\phi$ to $X/{{\rm Im}\phi}$. Observe that the corresponding counit $ji\stackrel{\sim}\longrightarrow {\rm Id}_{A\mbox{-mod}}$ is an isomorphism .  One checks that the cohomological dimension (\cite[p.57]{Har66}) of $j$ is at most one. In particular, the left derived functor $\mathbf{L}^b j\colon \mathbf{D}^b(A[M]\mbox{-mod})\rightarrow  \mathbf{D}^b(A\mbox{-mod})$ is defined. Moreover, we have the adjoint pair $(\mathbf{L}^b j, i_*)$, and the counit is an isomorphism.
Since the functor $j$ sends projective modules to projective modules, the functor $\mathbf{L}^b j$ preserves perfect complexes.
Then it induces a triangle functor ${\mathbf{L}^b \bar{j}}\colon \mathbf{D}_{\rm sg}(A[M])\rightarrow \mathbf{D}_{\rm sg}(A)$.
Moreover, we have the induced adjoint pair $({\mathbf{L}^b\bar{j}}, \bar{i_*})$, whose counit $({\mathbf{L}^b\bar{j}})\bar{i_*}\stackrel{\sim}\longrightarrow {\rm Id}_{\mathbf{D}_{\rm sg}(A)}$ is an isomorphism; see \cite[Lemma 1.2]{Or04}. In particular, the functor $\bar{i_*}$ is fully faithful.

It remains to show the denseness of $\bar{i_*}$. We now view the essential image ${\rm Im}\; \bar{i_*}$ of $\bar{i_*}$
as a full triangulated subcategory of $\mathbf{D}_{\rm sg}(A[M])$. It suffices to show
that for  each $A[M]$-module  $\begin{pmatrix} X \\ V\end{pmatrix}$, its image in  $\mathbf{D}_{\rm sg}(A[M])$
 lies in ${\rm Im}\; \bar{i_*}$; see Lemma \ref{lem:stalk}. Observe that $\Omega (\begin{pmatrix}0 \\ V\end{pmatrix})$ lies in ${\rm Im}\; i$, and then by
Lemma \ref{lem:exactsequence}, $q(\begin{pmatrix} 0 \\ V\end{pmatrix})$ lies in  ${\rm Im}\; \bar{i_*}$. The following
natural exact sequence induces a triangle in $\mathbf{D}_{\rm sg}(A[M])$
$$0\longrightarrow i(X)\longrightarrow \begin{pmatrix} X \\ V\end{pmatrix}\longrightarrow \begin{pmatrix} 0 \\ V\end{pmatrix}\longrightarrow 0.$$
This triangle implies that $q(\begin{pmatrix} X \\ V\end{pmatrix})$ lies in ${\rm Im}\; \bar{i_*}$. Then we are done.
\end{proof}

Let $D$ be a simple artin algebra, and let $_DN_A$ be a $D$-$A$-bimodule, on which $R$ acts centrally.
The \emph{one-point coextension} of $A$ by $N$ is the upper triangular matrix algebra $[N]A=\begin{pmatrix} D & N\\ 0 & A\end{pmatrix}$. A left $[N]A$-module is written as  $\begin{pmatrix} V \\ X\end{pmatrix}_\phi$, where $_DV$ and $_AX$
are a left $D$-module and $A$-module, respectively, and that $\phi\colon  M \otimes_A X\rightarrow V$ is a morphism of
$D$-modules. The valued quiver $Q_{[N]A}$ is obtained from $Q_A$ by adding a sink together with some valued arrows
ending at it, where the sink corresponds to the simple projective $[N]A$-module $\begin{pmatrix} S\\ 0\end{pmatrix}$
for a simple $D$-module $S$.

For the one-point coextension $[N]A$, we have an exact embedding $i\colon A\mbox{-mod}\rightarrow [N]A\mbox{-mod}$,
which sends $_AX$ to $i(X)=\begin{pmatrix} 0\\ X\end{pmatrix}$.

The following result is similar to Proposition \ref{prop:ope}, while the proof is simpler. This result
is closely related to \cite[Theorem 4.1(1)]{Ch09}.

\begin{prop}\label{prop:opce} Let $[N]A$ be the one-point coextension as above. Then the embedding $i\colon A\mbox{-{\rm mod}}\rightarrow [N]A\mbox{-{\rm mod}}$ induces a triangle equivalence
$$\mathbf{D}_{\rm sg}(A)\simeq \mathbf{D}_{\rm sg}([N]A).$$
\end{prop}

\begin{proof}
We observe that $i(A)$ has projective dimension at most one. Then the obviously induced functor $i_*\colon \mathbf{D}^b(A\mbox{-mod}) \rightarrow \mathbf{D}^b([N]A\mbox{-mod})$ preserves perfect complexes, and it induces the required functor $\bar{i_*}\colon \mathbf{D}_{\rm sg}(A)\rightarrow \mathbf{D}_{\rm sg}([N]A)$.

The functor $i$ admits an exact left adjoint $j\colon [N]A\mbox{-mod}\rightarrow A\mbox{-mod}$, which sends $\begin{pmatrix} V\\X \end{pmatrix}$ to $X$; moreover, $j$ preserves projective modules. Then it induces a triangle functor $\bar{j_*}\colon \mathbf{D}_{\rm sg}([N]A)\rightarrow \mathbf{D}_{\rm sg}(A)$, which is
left adjoint to $\bar{i_*}$. Then as in the proof of Proposition \ref{prop:ope}, we have  that $\bar{i_*}$ is fully faithful. The denseness of $\bar{i_*}$  follows from the natural exact sequence
$$0 \longrightarrow \begin{pmatrix} V\\ 0\end{pmatrix}\longrightarrow \begin{pmatrix} V \\ X\end{pmatrix}\longrightarrow i(X)\longrightarrow 0,$$
for each $[N]A$-module $\begin{pmatrix} V \\ X\end{pmatrix}$, and that the module $\begin{pmatrix} V\\ 0\end{pmatrix}$ is projective.  We omit the details.
\end{proof}

We use the above two propositions to reduce the study of the singularity category of arbitrary artin algebras
to cyclic-like ones.

Let $A$ be an artin algebra. Consider the valued quiver $Q_A$. A vertex $e$ is called \emph{cyclic} provided that
there is an oriented cycle containing it, and the corresponding simple $A$-module is called \emph{cyclic}. More generally,  a vertex $e$ is called \emph{cyclic-like} provided that there is a path through $e$,
which starts with a cyclic vertex and ends at a cyclic vertex, while the corresponding simple
$A$-module is  called \emph{cyclic-like}.   An artin algebra $A$ is called \emph{cyclic-like} provided that
its valued quiver $Q_A$ is cyclic-like. This is equivalent to that $A$ has neither simple projective nor simple injective modules.

For an artin algebra $A$, its \emph{cyclicization} is an artin algebra $A_c$ which is either simple or cyclic-like, such that
there is a sequence $A_c=A_0, A_1, \cdots, A_r=A$ with each $A_{i+1}$ is a one-point (co)extension of $A_i$.

The following is an immediate consequence of the definition.

\begin{lem}\label{lem:cyclic}
Let $A$ be an artin algebra with its cyclicization $A_c$. Then we have a triangle equivalence
$$\mathbf{D}_{\rm sg}(A_c)\simeq \mathbf{D}_{\rm sg}(A). $$
\end{lem}

\begin{proof}
Apply Propositions \ref{prop:ope} and \ref{prop:opce}, repeatedly.
\end{proof}

The following result seems to be well known.

\begin{lem} The following statements hold.
\begin{enumerate}
\item Each artin algebra has a cyclicization.
\item Let $A_c$ and $A_{c'}$ be two cyclicizations of $A$. Then if $A_c$ is simple, so is $A_{c'}$. Otherwise,
we have an isomorphism  $A_c\simeq A_{c'}$ of algebras.
\end{enumerate}
\end{lem}

\begin{proof}
(1) It follows from the well-known fact that the existence of a simple injective (\emph{resp.} projective) module of $A$ implies that $A$ is a one-point extension (\emph{resp.} coextension) of $A'$. Moreover, the valued quiver $Q_{A'}$ of $A'$ is obtained from the one of $A$ by deleting the relevant source (\emph{resp.} sink).

(2) The first statement follows from the observation that passing from $A$ to $A'$ in (1),  the set of cyclic-like
vertices stays the same.

For the isomorphism of algebras, it suffices to observe that $A_c\mbox{-mod}$ is equivalent to the smallest
Serre subcategory (\cite[Chapter 15]{Fai73}) of $A\mbox{-mod}$ containing the cyclic-like simple $A$-modules $S$; moreover, the multiplicity of $P_{A_c}(S)$ in the indecomposable decomposition of $A_c$ equals the multiplicity of $P(S)$ in the one of
$A$. Here, $P(S)$ and $P_{A_c}(S)$ denote the projective cover of $S$ as an $A$-module and $A_c$-module, respectively.
\end{proof}

%


\section{Hom-finiteness of  singularity categories}

In this section, we study the Hom-finiteness of the singularity category of an artin algebra with radical
square zero, and  prove Theorem C.

Throughout, $A$ is an artin $R$-algebra such that
its Jacobson radical $\mathbf{r}$ satisfies $\mathbf{r}^2=0$. Recall that in this case, the syzygy $\Omega(X)$
of any module $X$ is semisimple.

\begin{lem}\label{lem:injective}
Suppose that $A$ is cyclic-like. Then we have
\begin{enumerate}
\item each simple $A$-module has infinite projective dimension;
\item for each $i\geq 0$, the algebra homomorphism ${\rm End}_{A/\mathbf{r}}(\mathbf{r}^{\otimes i})\rightarrow {\rm End}_{A/\mathbf{r}}(\mathbf{r}^{\otimes i+1})$ induced by $\mathbf{r}\otimes_{A/\mathbf{r}}-$ is injective.
\end{enumerate}
\end{lem}

 \begin{proof}
 (1) Recall that a cyclic-like algebra does not have simple projective modules. Then the statement follows
 from the observation that for a simple module $S$ with finite projective dimension, we have that ${\rm proj.dim}\; \Omega(S)={\rm proj.dim}\; S-1$.

 (2) We recall that $A\mbox{-\underline{ssmod}}$ is the full subcategory of $A\mbox{-\underline{mod}}$ consisting of semisimple modules. Then by (1),  $A\mbox{-\underline{ssmod}}$ is naturally equivalent to $A\mbox{-ssmod}$, and the syzygy functor $\Omega\colon A\mbox{-\underline{ssmod}}\rightarrow A\mbox{-\underline{ssmod}} $ is faithful. Now the result follows from Lemma \ref{lem:omega}.
 \end{proof}

Recall that the singularity category $\mathbf{D}_{\rm sg}(A)$ is naturally $R$-linear. We are interested in the problem when it is Hom-finite, that is, all the Hom sets are finitely generated $R$-modules.

\begin{thm}\label{thm:C}
Let $A$ be an artin algebra with radical square zero. Then the following statements are equivalent:
\begin{enumerate}
\item the singularity category $\mathbf{D}_{\rm sg}(A)$ is Hom-finite;
\item the associated regular algebra $\Gamma(A)$ is semisimple;
\item the cyclicization $A_c$ of $A$ is either simple or isomorphic to a finite product of
self-injective algebras.
\end{enumerate}
\end{thm}

We point out that the cyclicization $A_c$ of $A$ is necessarily with radical square zero.
Recall that an indecomposable non-simple artin algebra with radical square zero is   self-injective
if and only if its valued quiver is an oriented cycle with the trivial valuation;
see \cite[Proposition IV.2.16]{ARS} or the proof of \cite[Corollary 1.3]{Ch11}. Then the statement (3)
above is equivalent to the corresponding one in Theorem C.

\begin{proof} Recall from Proposition \ref{prop:K^n(A)} the isomorphism  $\Gamma(A)\simeq {\rm End}_{\mathbf{D}_{\rm sg}(A)}(q(A/\mathbf{r}))$.
Then we have the implication ``(1)$\Rightarrow$(2)", since an artin regular algebra is necessarily semisimple.

For ``(2)$\Rightarrow$(1)", consider the cyclicization $A_c$ of $A$, whose Jacobson radical
is denoted by $\mathbf{r}_c$.  Then by Lemma \ref{lem:cyclic} we have an equivalence
 $\mathbf{D}_{\rm sg}(A_c)\simeq \mathbf{D}_{\rm sg}(A)$. Applying Proposition \ref{prop:sg} we have an equivalence
 ${\rm proj}\; \Gamma(A_c)\simeq {\rm proj}\; \Gamma(A)$, that is, $\Gamma(A_c)$ and $\Gamma(A)$ are Morita equivalent.
 Then $\Gamma(A_c)$ is also semisimple. Recall that $\Gamma(A_c)=\varinjlim {\rm End}_{A_c/\mathbf{r}_c}(\mathbf{r}_c^{\otimes i})$. By Lemma \ref{lem:injective}
all the canonical maps ${\rm End}_{A_c/\mathbf{r}_c}(\mathbf{r}_c^{\otimes i})\rightarrow \Gamma(A_c)$
are injective. Recall that for a semisimple algebra, the number of pairwise orthogonal idempotents is bounded. Then the $R$-lengths  of the algebras ${\rm End}_{A_c/\mathbf{r}_c}(\mathbf{r}_c^{\otimes i})$ are uniformly bounded. Consequently, the algebra $\Gamma(A)$ is an artin $R$-algebra. By Proposition \ref{prop:sg} the singularity category $\mathbf{D}_{\rm sg}(A_c)$ is Hom-finite. Then we are done by  Lemma \ref{lem:cyclic}.

Recall from \cite[Theorem 2.1]{Ric} that the singularity category of a self-injective algebra is equivalent to its stable category. In particular, it is Hom-finite. Then the implication ``(3)$\Rightarrow$(1)" follows from Lemma \ref{lem:cyclic}.

It remains to show ``(1)$\Rightarrow$(3)". Without loss of generality, we assume that the algebra $A$ is cyclic-like
such that $\mathbf{D}_{\rm sg}(A)$ is Hom-finite. We will show that $A$ is self-injective.

We claim that the sysygy $\Omega(S)$ of any cyclic simple $A$-module $S$ is simple. Then
there is only one arrow starting at $S$ in $Q_A$, which is valued by $(1, b)$ for some natural number $b$. Since $Q_A$ is cyclic-like, this forces that $Q_A$ is a disjoint union of oriented cycles. In each oriented cycle, every arrow has valuation
$(1, b_i)$ for some $b_i$. Then the symmetrization condition implies that all these $b_i$'s are necessarily  one; compare the proof of \cite[Proposition VIII. 6.4]{ARS}. As we point out above, this implies that $A$ is self-injective.

 We prove the claim. Since by Corollary \ref{cor:idempotent} $\mathbf{D}_{\rm sg}(A)$ is idempotent split,
we have that $\mathbf{D}_{\rm sg}(A)$ is a Krull-Schmidt category (\cite[Appendix A]{CYZ}). In particular, each object is uniquely decomposed as a direct sum of finitely many indecomposable objects. We observe that for each semisimple module $X$,  $lX\leq l\Omega(X)$. Here, $l$ denotes the composition length. Consider a cyclic simple $A$-module $S$, and take a path $S=S_1\rightarrow S_2\rightarrow \cdots \rightarrow S_r \rightarrow S_{r+1}=S$ in $Q_A$.  Assume on the contrary that $l\Omega(S)\geq 2$. Then we have  $\Omega(S)=S_2\oplus X$ for some nonzero semisimple module $X$. Observe that $S$ is a direct summand of $\Omega^{r-1}(S_2)$, and
then we have $\Omega^r(S)=S\oplus X'$ for a nonzero semisimple module $X'$. Consequently, we have $\Omega^{nr}(S)=S\oplus X'\oplus\Omega^r(X')\oplus  \cdots \oplus  \Omega^{(n-1)r}(X')$. Then the lengths of the semisimple modules $\Omega^{nr}(S)$
tend to the infinity, when $n$ goes to the infinity. By Lemma \ref{lem:injective}(1),  $q(T)$ is not zero for any simple $A$-module $T$. Recall from  Lemma \ref{lem:exactsequence} that $q(S)\simeq q(\Omega^{nr}(S))[nr]$. This contradicts to the Krull-Schmidt property of $q(S)$, and we are done with the claim. \end{proof}

\vskip 5pt

 \noindent {\bf Acknowledgements.}\quad The author thanks Professor Zhaoyong Huang and Professor Yu Ye for their helpful comments.

\bibliography{}

\vskip 10pt

 {\footnotesize \noindent Xiao-Wu Chen, Department of
Mathematics, University of Science and Technology of
China, Hefei 230026, Anhui, PR China \\
Wu Wen-Tsun Key Laboratory of Mathematics, USTC, Chinese Academy of Sciences, Hefei 230026, Anhui, PR China.\\
URL: http://mail.ustc.edu.cn/$^\sim$xwchen}

\end{document}